\documentclass[12pt]{article}

\usepackage{latexsym,amssymb,amsgen,amsmath,amsxtra, amsfonts, amscd}

\usepackage{authblk}

\usepackage{mathrsfs}
\usepackage{amsthm}

\newtheorem{theorem}{Theorem}

\newtheorem{remark}[theorem]{Remark}

\newtheorem{definition}{Definition}

\newcommand{\R}{{\mathbb{R}}}

\newcommand{\Q}{{\mathbb{Q}}}

\def\be{\begin{equation}}
\def\ee{\end{equation}}
\def\ba{\begin{array}}
\def\ea{\end{array}}
\def\vp{\varphi}

\title{Measurable process selection theorem and non-autonomous inclusions}
\author{Jorge E. Cardona and Lev Kapitanski}
\affil{ 
Department of Mathematics, University of Miami, Coral Gables, FL 33124, USA}

\date{}

\begin{document} 

\maketitle

\begin{abstract} 
A semi-process is an analog of the semi-flow for non-autonomous differential equations or inclusions. 
We prove an abstract result on the existence of measurable semi-processes in the situations where 
there is no uniqueness. Also, we allow solutions to blow up in finite time and then obtain local semi-processes.    
\end{abstract}


\section{Introduction}

Let 
\be\label{aut}
\frac{du}{dt} = f(u)
\ee
be an archetypical autonomous differential equation. Autonomous refers to the structure of the equation and means that the independent variable, $t$, does not appear explicitly (independently) in the equation. 
Because of that, \eqref{aut} is invariant under the time shift (translation) $\theta_\tau: t\mapsto t+\tau$, and if $u(\cdot)$ is a solution of \eqref{aut}, then $\theta_\tau u(\cdot) = u(\cdot + \tau)$ is a solution as well. Suppose \eqref{aut} describes evolution/dynamics on some set $X$ (which could be a finite- or infinite-dimensional vector space or manifold). Given an $a\in X$, let $u(t, a)$, $t\in [0, +\infty)$ be a solution of \eqref{aut} starting at $u(0, a) = a$ (let us assume that global solutions exist forward in time). If $v(\cdot, u(t_1, a))$ is a solution of \eqref{aut} starting (at $t = 0$) from the point $u(t_1, a)$, we 
can splice $u$ and $v$ and obtain a (possibly new) solution $w = u \underset{t_1}{\bowtie} v$ starting at $a$:
\be
w(t) = u \underset{t_1}{\bowtie} v(t) = \begin{cases}
u(t, a)\,, & \text{if}\; 0\le t\le t_1, \\ 
v(t - t_1, u(t_1, a))\,, & \text{if}\; t\ge t_1\,.
\end{cases}
\ee
If solutions are unique (for every $a\in X$ there is a unique solution $u(t, a)$), then 
$v(t, u(t_1, a)) = u(t + t_1, a)$. In general, in the case of uniqueness, the solutions of \eqref{aut} 
enjoy  
the semigroup property, i.e., for every $a\in X$, $u(0, a) = a$ and 
\be\label{sg}
u(t_2, u(t_1, a)) = u(t_1 + t_2, a)\,,\quad \forall t_1, t_2 \ge 0\,.
\ee
This allows us to define the semigroup $U(t): X\to X$ by the formula $U(t)(a) = u(t, a)$.

For non-autonomous differential equations the situation is similar and different. Consider 
an archetypical non-autonomous equation 
\be\label{n-aut}
\frac{du}{dt} = g(t, u)\,.
\ee
Now, in addition to the initial position/state $a\in X$, it is important to specify the initial moment of time, $t_0$. The solution(s) will depend on $a$ and $t_0$; we write $u(t; t_0, a)$ to denote a solution of \eqref{n-aut} for $t\ge t_0$ that equals $a$ when $t = t_0$. If we follow the solution 
$u(t; t_0, a)$ until the moment $t = t_1$ and then follow a solution $v(t; t_1, u(t_1; t_0, a))$ that 
starts at the point $u(t_1; t_0, a)$ at the moment $t_1$, we obtain a spliced solution of \eqref{n-aut}, 
\be\label{splice}
w(t; t_0, a) = u \underset{t_1}{\bowtie} v(t) = \begin{cases}
u(t; t_0,  a)\,, & \text{if}\; t_0\le t\le t_1, \\ 
v(t; t_1, u(t_1;, t_0,  a))\,, & \text{if}\; t\ge t_1\,.
\end{cases}
\ee
If the solutions of equation \eqref{n-aut} are unique, we have the following analog of the semigroup property:  
\be\label{n-sg}
u(t_0; t_0, a) = a \quad\text{and}\quad u(t_2; t_1, u(t_1; t_0, a)) = u(t_2; t_0, a)\,,\quad \forall a\in X\;\forall t_2\ge t_1\ge t_0\ge 0\,.
\ee
Also, we can define the \textit{transition} map $U(t_1; t_0)$ that maps $X$ into $X$ by assigning to every $a\in X$ the 
value $u(t_1; t_0, a)$ of the solution $u(t; t_0, a)$. This transition map has the properties 
\begin{subequations}\label{NSG:main}
\begin{align}
& U(t; t) = \hbox{id}_X\,,\;\forall t\ge 0 \label{NSG:a}\\
& U(t_2; t_1)\circ U(t_1; t_0) = U(t_2; t_0)\,,\;\forall t_2\ge t_1\ge t_0\ge 0\,.\label{NSG:b}
\end{align}
\end{subequations}
A family of maps $U(t_1; t_0)$ with the properties \eqref{NSG:main} will be called a \textit{process} 
(see, e.g., \cite{Dafermos, Hale}). To the autonomous case correspond 
homogeneous processes characterized by time invariance: $U(t_1-\tau; t_0 - \tau) = U(t_1; t_0)$ for all 
(admissible) $\tau$ (and therefore $U(t) = U(t; 0)$ is the semigroup). 
 
It may be advantageous to think of solutions of the autonomous equation \eqref{aut} as integral curves (trajectories) in $X$, 
i.e., view solitons as continuous (infinite, one-sided) paths in $X$. If there is no uniqueness, the solutions/integral curves starting at the point $a$ form an integral funnel, $S(a)$, a subset of all paths starting at $a$. (Analysis of integral funnels for ODEs was initiated by H. Kneser in the 1920s, \cite{Kneser}.) 
Denote by $\Omega$ the space of all continuous (infinite, one-sided) paths in  $X$. Then the map 
$a\mapsto S(a)$ is a set-valued map (other names: multifunction, correspondence) from $X$ into $2^\Omega$. It has the (already mentioned) properties: if the path $w$ is in $S(a)$, then its shift, $\theta_\tau w$ is in $S(w(\tau))$, and if $u\in S(a)$ and $v\in S(u(t_1))$, then $u \underset{t_1}{\bowtie} v\in S(a)$.  

An interesting and important question is whether it is possible to select a solution $u(\cdot, a)$ from 
every funnel $S(a)$ in such a way, that $u(t, a)$ has the semigroup property \eqref{sg}. In other words, is it possible to define a semigroup (semiflow) $U(t)$ so that, for every $a\in X$, $u(t, a) = U(t)(a)$, $t\ge 0$,  is a path in $S(a)$? In \cite{C,C-K}, we show that the answer is 
yes under some very general assumptions. Moreover, we show that the selection $a\mapsto u(\cdot, a)\in S(a)$ is measurable. This is a new type of selection theorems (for measurable selection results see, e.g., \cite[Section 18.3]{A-B}, \cite{H-P}, and the surveys \cite{Wa, Io}). Our results were 
motivated by the Markov selection theorems, see \cite{Krylov,S-V,F-R,G-R-Z}.

Here, we extend the semiflow selection theorem of \cite{C,C-K} to non-autonomous equations and processes. 
If there is no uniqueness for \eqref{n-aut}, but the solutions to the initial-value problem exist forward in time, 
we have the integral funnels $S(t_0, a)$ formed by all the solution $u(t; t_0, a)$, $t\ge t_0$, 
such that 
$u(t_0; t_0, a) = a$. The question we ask is whether there exists a measurable selection 
$(t_0, a) \mapsto u(\cdot; t_0, a)\in S(t_0, a)$ such that $u(t; t_0, a)$ satisfies \eqref{n-sg}. 
The answer again is yes under the right assumptions. After our previous work \cite{C, C-K}, this answer is not surprising if we are concerned with solutions of the non-autonomous equation \eqref{n-aut}: one could 
replace \eqref{n-aut} with the equivalent autonomous system
\be\label{n-aut-2}
\frac{du}{ds} = g(t, u)\,,\quad \frac{dt}{ds} = 1\,.
\ee
Then, if we apply our results from \cite{C,C-K} and obtain a semigroup ${\tilde U}(s)$, $s\ge 0$, on $X\times [0, +\infty)$ corresponding to \eqref{n-aut-2}, the maps 
$U(t_1, t_0)(a) = {\tilde U}(t_1)(a, t_0)$ would form a process corresponding to \eqref{n-aut}. However, 
if we are in a more general setting and deal with integral funnels, their \textit{autonomisation} is not clear.  
Thus, 
in the next section we present the precise statement and the proof of the existence of a measurable process. After that, in Section~\ref{loc-proc} we study the existence of \textit{local processes}. These 
apply to the situations, where in addition to non-uniqueness we allow solutions to blow up in finite time. [In the 1960--70s there was some interest in abstracting the dynamical and semi-dynamical systems and processes, see, e.g., \cite{Hajek-1,Hajek-2,B-H}. Our approach involves  funnels and is different.]
An example on the semi-process selection is presented in Section~\ref{sec:ex}.  

A different and important point of view on non-autonomous dynamics involves the skew-product construction, see, e.g., \cite{S-S,S}. The skew-product set-up is very convenient for the study of the long term behavior of non-autonomous systems. 
However, for the measurable processes selection, we believe our direct approach is more natural. 
We should mention that 
our abstract results apply not only to non-autonomous ordinary differential equations, but to partial differential equations, to differential and difference inclusions, and to other situations where 
integral funnels make sense. 

\section{Global processes}\label{glob-proc}

Let $X$ be a separable complete metric space with metric $\rho$, which we assume to be bounded: 
$\rho(x, y)\le 1$ for all $x, y\in X$. Denote by ${\cal B}_X$ the Borel $\sigma$-algebra of $X$. 
Let $\Omega$  be the space of all continuous infinite one-side paths in $X$ equipped with the compact-open topology. The elements of $\Omega$ 
are continuous maps $u: [0, +\infty)\to X$, and convergence in $\Omega$ is the uniform convergence on every finite time interval. The space $\Omega$ is Polish; we fix a complete bounded metric on $\Omega$ by setting 
\be\label{met-Om}
d(u, v) = \sum_{\ell = 1}^\infty 2^{-\ell}\,\sup_{t\in [0, \ell]}\rho(u(t), v(t))\;[1 + \sup_{t\in [0, \ell]}\rho(u(t), v(t))]^{-1}\,.
\ee
Sometimes it is convenient  
to view the paths parametrized by $s\in [\tau, +\infty)$. We denote by $\Omega^\tau$ the space of continuous maps from $[\tau, +\infty)$ into $X$. It can be identified with the image of $\Omega$ under
the (past erasing) map $\sigma_\tau$: if $u(t)$, $t\ge 0$, is a path in $\Omega$, then the path $\sigma_\tau u : [\tau, +\infty)\to X$ is defined as $\sigma_\tau u(t) = u(t)$ for $t\ge \tau$. On the other hand, $\Omega^\tau$ can be viewed as a subset of $\Omega$ if we extend the paths $v(t)$ in  
$\Omega^\tau$ to $[0, \tau)$ as staying at $v(\tau)$:
\[
\Game_\tau : \Omega^\tau \to \Omega\,,\quad (\Game_\tau v)(t) = \begin{cases} 
v(\tau) & \text{when} 0\le t < \tau\,, \\ 
v(t) & \text{when} t \ge \tau\,. 
\end{cases}
\]
We will also use the notation $\Omega^\tau_a$ for all the paths in $\Omega^\tau$ staring at the point $a$, i.e., $v(\tau) = a$ if $v\in \Omega^\tau_a$. On occasion, it will be convenient to specify $\tau$ and $a$ in the notation for a path, e.g.,  
$v(t; \tau, a)$.

\bigskip

The integral funnels $S(t_0, a)$ will be subsets of the set  $\Omega^{t_0}_a$. Denote by $P_{cl}[\Omega]$ the space of all (bounded) closed subsets of $\Omega$ endowed with the Vietoris topology (= exponential topology), see \cite{Kuratowski}, \cite{B}, and \cite{A-B} for details on set-valued maps, their properties and, in particular, measurability. In this presentation, all set-valued maps from $X$ to $\Omega$ will have  non-empty closed values and will be viewed 
as maps from the measurable space $(X, {\cal B}_X)$ to $P_{cl}[\Omega]$ with the Vietoris topology. 
If $\Gamma$ is such a map and $A\subset \Omega$, define 
\[
\Gamma^{-}(A) = \{x\in X : \Gamma(x)\cap A \neq\emptyset\}\,.
\]
For historical reasons, there are several confusingly similar notions of measurability of set-valued maps. 
A map $\Gamma: X\to P_{cl}[\Omega]$ is \textit{weakly measurable} if $\Gamma^{-}(G)\in {\cal B}_X$ for every open set $G\subset\Omega$. $\Gamma$ is \textit{measurable} if $\Gamma^{-}(F)\in {\cal B}_X$ for every closed set $F\subset\Omega$. Since $\Omega$ is metric, if $\Gamma$ is measurable, it is weakly measurable, \cite[Lemma 18.2]{A-B}. Since $\Omega$ is in addition separable, if $\Gamma$ is \textit{compact-valued} and weakly measurable, it is measurable, \cite[Theorem 18.10]{A-B}. Thus, 
in our setting, 
for compact-valued $\Gamma: X\to P_{cl}[\Omega]$, there is no difference between weak measurability and 
measurability. 
\medskip

The fundamental result of Kuratowski and Ryll-Nardzewski, \cite[Theorem 1, p. 398]{K-R}, implies that if 
$\Gamma: X\to P_{cl}[\Omega]$ is weakly measurable, then $\Gamma$ has a measurable selection, i.e., 
there exists a single-valued map $\gamma : X\to \Omega$ such that $\gamma(x)\in \Gamma(x)$ for all $x\in X$ and 
$\gamma$ is $({\cal B}_X, {\cal B}_\Omega)$-measurable: for every Borel set $A\subset \Omega$, 
$\gamma^{-1}(A)$ is a Borel set in $X$. 

\bigskip

We introduce now abstract integral funnels $S(t_0, a)$ that have the properties prompted by the properties of integral funnels  
of solutions of equation \eqref{n-aut}. 

\begin{definition} 
$S(t_0, a)$, where $t_0\in [0, +\infty)$ and $a\in X$,  is a family of abstract integral funnels on the space $X$ if, 
for every $t_0\ge 0$, $S(t_0, \cdot) : X\to P_{cl}[\Omega]$ is a set-valued map 
with the following properties. 
\begin{enumerate}
\item[{\bf S1}] For every $a\in X$, $S(t_0, a)$ is a non-empty compact subset of $\Omega^{t_0}_a$. 
Every path $u$ in $S(t_0, a)$ is parametrized as $u(t; t_0, a)$, where $t\ge t_0$.
\item[{\bf S2}] Each map $S(t_0, \cdot)$ is measurable, i.e., for every closed set $C\subset \Omega^{t_0}$, 
\[
\{x\in X: S(t_0, x)\cap C \neq \emptyset\}\in {\cal B}_X\,.
\]
\item[{\bf S3}] If $u\in S(t_0, a)$, then $\sigma_\tau u \in S(t_0+\tau, u(t_0+\tau))$. 
\item[{\bf S4}] If $u\in S(t_0, a)$ and $v\in S(t_0+\tau, u(t_0+\tau))$, then the spliced path 
$w = u \underset{t_0+\tau}{\bowtie} v$, defined as in \eqref{splice}, belongs to the funnel 
$S(t_0, a)$.
\end{enumerate}
\end{definition}

\begin{theorem}\label{thm:main} 
Every family of abstract integral funnels $S(t_0, a)$, $t_0\ge 0$, $a\in X$, has, for every $t_0$, a measurable selection 
$ a\to u(\cdot; t_0, a)\in S(t_0, a)$ with the semigroup property \eqref{n-sg}. 
As a corollary, there is a Borel measurable (semi)process $U(t_1, t_0) : X\to X$ whose orbits are $U(t_1, t_0)(a) = u(t_l; t_0, a)$, for all $t_1 \ge t_0\ge 0$ and for all $a\in X$.
\end{theorem}
\begin{proof} We modify our proof for the autonomous case from \cite{C-K}. The plan is to 
successively reduct each funnel $S(t_0, a)$ while preserving the properties {\bf S1 - S4} and so that the limiting funnel would contain just one path; 
property {\bf S3} for the limiting funnel then spells \eqref{n-sg}. 
The reduction of the funnels is based on an idea from optimization theory that was used by N. V. Krylov 
in his proof of the Markov selection theorem, \cite{Krylov}. 
\medskip

Let $\vp: X\to [0, 1]$ be a continuous function and let $\lambda$ be a positive real number. 
For $t_0 \ge 0$ and $a\in X$, define the functional $\zeta$ on $S(t_0, a)$ via the formula
\be\label{zeta}
\zeta(w) = \int_0^\infty e^{-\lambda t} \vp(w(t_0 + t))\,dt\,.
\ee
This is a continuous functional and it attains its maximum on the compact set $S(t_0, a)$. 
Denote this maximum by $m_\zeta(t_0, a)$, 
\be\label{m}
m_\zeta(t_0, a) = \max_{w\in S(t_0, a)} \zeta(w)\,
\ee
[The function $m_\zeta(t_0, \cdot): X \to \R$ is called the value function.] Define
\be\label{V}
V_\zeta[S(t_0, a)] = \{v\in S(t_0, a):\; \zeta(v) = m_\zeta(t_0, a)\}\,.
\ee
By the so-called measurable maximum theorem, \cite[Theorem 18.19]{A-B}, $V_\zeta[S(t_0, a)]$ is a non-empty compact subset of $S(t_0, a)$, and the set-valued map $a \mapsto V_\zeta[S(t_0, a)]$ is measurable. Thus, the family of sets $V_\zeta[S(t_0, a)]$ has properties {\bf S1} and {\bf S2} of 
the abstract funnels. Let us check that it has the remaining two properties {\bf S3} and {\bf S4}. 
Suppose $u\in V_\zeta[S(t_0, a)]$ and consider the shifted path $\sigma_\tau u$. By property 
{\bf S3} of the family $S$, $\sigma_\tau u\in S(t_0+\tau, u(t_0 + \tau))$. We have to show that 
$\sigma_\tau u$ maximizes $\zeta$ in the set $S(t_0+\tau, u(t_0 + \tau))$. Pick any path $v$ in 
$S(t_0+\tau, u(t_0 + \tau))$ and consider the spliced path $w = u \underset{t_0+\tau}{\bowtie} v$, which, by property {\bf S4}, belongs to the funnel 
$S(t_0, a)$. Since $u$ maximizes $\zeta$ over  $S(t_0, a)$, $\zeta(u) \ge \zeta(w)$, i.e., 
\[
\int_0^\infty e^{-\lambda t} \vp(u(t_0 + t))\,dt \ge \int_0^\infty e^{-\lambda t} \vp(w(t_0 + t))\,dt\,.
\]
But 
\[
\begin{aligned}
& \int_0^\infty e^{-\lambda t} \vp(w(t_0 + t))\,dt = \int_0^\tau e^{-\lambda t} \vp(u(t_0 + t))\,dt + 
\int_\tau^\infty e^{-\lambda t} \vp(v(t_0 + t))\,dt = \\ 
& \int_0^\tau e^{-\lambda t} \vp(u(t_0 + t))\,dt + e^{-\lambda \tau}\,\int_0^\infty e^{-\lambda t} \vp(v(t_0 + \tau + t))\,dt\,,
\end{aligned}
\]
while
\[
\int_0^\infty e^{-\lambda t} \vp(u(t_0 + t))\,dt = \int_0^\tau e^{-\lambda t} \vp(u(t_0 + t))\,dt + e^{-\lambda \tau}\,\int_0^\infty e^{-\lambda t} \vp(u(t_0 + \tau + t))\,dt\,.
\]
Hence, 
\[
\int_0^\infty e^{-\lambda t} \vp(u(t_0 + \tau + t))\,dt \ge \int_0^\infty e^{-\lambda t} \vp(v(t_0 + \tau + t))\,dt\,,
\]
which means $\sigma_\tau u$ is a maximizer in $S(t_0+\tau, u(t_0 + \tau))$. 
To check property {\bf S4} for $V_\zeta[S]$, pick $u\in V_\zeta[S(t_0, a)]$ and $v\in V_\zeta[S(t_0+\tau, u(t_0+\tau))]$  and consider the spliced path 
$w = u \underset{t_0+\tau}{\bowtie} v$. We have to show that $w$ maximizes $\zeta$ over $S(t_0, a)$ and hence belongs to $V_\zeta[S(t_0, a)]$. This follows from a simple calculation that takes into account 
what we have just shown, that $\zeta(\sigma_\tau u) = m_\zeta(t_0 + \tau, u(t_0 + \tau))$. If 
$v\in V_\zeta[S(t_0+\tau, u(t_0+\tau))]$, then $\zeta(v) = m_\zeta(t_0 + \tau, u(t_0 + \tau))$ as well. Thus,
\[
\begin{aligned}
& \int_0^\infty e^{-\lambda t} \vp(w(t_0 + t))\,dt = \int_0^\tau e^{-\lambda t} \vp(u(t_0 + t))\,dt + 
\int_\tau^\infty e^{-\lambda t} \vp(v(t_0 + t))\,dt = \\ 
& \int_0^\tau e^{-\lambda t} \vp(u(t_0 + t))\,dt + e^{-\lambda \tau} \zeta(v) = 
\int_0^\tau e^{-\lambda t} \vp(u(t_0 + t))\,dt + e^{-\lambda \tau} \zeta(\sigma_\tau u) = \zeta(u)\,,
\end{aligned}
\]
i.e., $\zeta(w) = m_\zeta(t_0, a)$, i.e., $w\in V_\zeta[S(t_0, a)]$. To summarize, for any functional $\zeta$ of the form \eqref{zeta}, $V_\zeta[S(t_0, a)]$ is a family of abstract integral funnels.
\medskip

Now, choose a countable family $\Phi$ of continuous functions $\vp : X\to [0, 1]$ that separates the points of $X$, see \cite{B-K}. Choose some 
enumeration $(\lambda_n, \vp_n)$ of the countable set of pairs $(\lambda, \vp)$, where $\lambda$ runs through positive rational numbers and $\vp$ runs through $\Phi$. To each pair $(\lambda_n, \vp_n)$ 
corresponds the functional $\zeta_n$ via \eqref{zeta}. Define recursively the shrinking families of 
abstract integral funnels 
\[
S^0(t_0, a) = S(t_0, a), \quad S^{n}(t_0, a) = V_{\zeta_n}[S^{n-1}(t_0, a)], n = 1, 2, \dots\,.
\]
For each $(t_0, a)$, $S^{n}(t_0, a)$ is a sequence of nested compacta in $\Omega^{t_0}$.  
The intersection, 
\[
S^\infty(t_0, a) = \bigcap_{n = 0}^\infty S^{n}(t_0, a)\,,
\] 
is not empty and compact. In fact, $S^\infty(t_0, a)$ is an abstract family of integral funnels. Indeed,  
it is easy to see that properties {\bf S1}, {\bf S3}, and {\bf S4}, are satisfied. As the intersection of compact-valued maps into a Polish space, $S^\infty(t_0, \cdot)$ is measurable by \cite[Lemma 18.4]{A-B}. 
\medskip

It turns out that each funnel $S^\infty(t_0, a)$ is a singleton. Indeed, if $u, v\in S^\infty(t_0, a)$, then, for every $\vp\in \Phi$, 
\[
\int_0^\infty e^{-\lambda t} \vp(u(t_0 + t))\,dt = \int_0^\infty e^{-\lambda t} \vp(v(t_0 + t))\,dt\,,\quad \forall \lambda\in \Q_+\,.
\]
By the uniqueness of the Laplace transform, $\vp(u(t_0 + t)) = \vp(v(t_0 + t))$ for all $t\ge 0$. 
Because this is true for every $\vp\in\Phi$ and the family $\Phi$ separates the points of $X$, 
we obtain $u(t_0 + t) = v(t_0 + t)$ for all $t\ge 0$, i.e., $u = v$ as paths. 
\medskip

For $u\in S^\infty(t_0, a)$, we will use the notation $u(t; t_0, a)$, where $t\ge t_0$. If $t_1 > t_0$, then $(\sigma_{t_1 - t_0} u)(t) = u(t; t_1, u(t_1; t_0, a))$ for $t\ge t_1$, by property {\bf S3}.  
Thus, if $t_2 > t_1$, $(\sigma_{t_1 - t_0} u)(t_2) = u(t_2; t_1, u(t_1; t_0, a))$. On the other hand, 
$(\sigma_{t_1 - t_0} u)(t_2) = u(t_2; t_0, a)$ by the definition of the shift operator $\sigma_\tau$. 
This establishes the semigroup property of the measurable selection $u(t; t_0, a)$. 

Once the (measurable) selection $u$ is found, we define the process $U(t_1, t_0): X\to X$ for $t_1\ge t_0\ge 0$, by the formula $U(t_1, t_0)(a) = u(t_1; t_0, a)$. As we have shown, the map $a\mapsto \{u(\cdot; t_0, a)\}$ from $X$ to $2^\Omega$ is measurable and singleton-valued. 
By the Kuratowski-Ryll-Nardzewski selection theorem the map 
$a\mapsto u(\cdot; t_0, a)$ from $X$ to $\Omega^{t_0}$ is $({\cal B}_X,{\cal B}_\Omega)$-measurable. 
The map $U(t_1, t_0): X\to X$ is the composition of the map $a\mapsto u(\cdot; t_0, a)$ and 
the evaluation map $\pi_{t_1} : \Omega^{t_0}\owns w\to \pi_{t_1}(w) = w(t_1)\in X$, which is continuous. 
Consequently, $U(t_1, t_0)$ is Borel measurable.
 This
completes the proof.

\end{proof}

\section{Local processes}\label{loc-proc}

Let $X$ be a separable complete metric space as in the previous section. The integral funnels $S(t_0, a)$ will now be local. This means that to every initial state $a\in X$ and every initial moment $t_0$  corresponds a strictly positive number $T(t_0, a)$, the \textit{terminal time}. The paths in 
the funnel $S(t_0, a)$ form a subset in $C([t_0, t_0 + T(t_0, a))\to X)$. 
\medskip

\begin{definition}\label{loc-fun} 
A family $S(t_0, a)$, $t_0\in [0, +\infty)$, $a\in X$, will be called a family of abstract local  integral funnels with terminal times $T(t_0, a)$ 
if they satisfy the following conditions. 
\begin{enumerate}
\item[{\bf TT}] $T(t_0, a)$ is a lower semi-continuous function on $[0, +\infty)\times X$, i.e., 
if $(t_n, a_n)\to (t_0, a)$, then $T(t_0, a) \le \liminf T(t_n, a_n)$.
\item[{\bf LS1}] Every set $S(t_0, a)$ is a non-empty compact in the space $C([t_0, t_0 + T(t_0, a))\to X)$ with the topology of uniform convergence on every closed subinterval $[\alpha, \beta]$ of 
$[t_0, t_0 + T(t_0, a))$.
Every path $w(\cdot; t_0, a)\in S(t_0, a)$ is a continuous map from $[t_0, t_0 + T(t_0, a))$ into $X$, and $w(t_0; t_0, a) = a$. 
\item[{\bf LS2}] For every $t_0\ge 0$, the set-valued map $a\mapsto S(t_0, a)$ is measurable in the following sense. Each path $w$ in $S(t_0, a)$ can be re-parametrized as ${\tilde w}(s) = w(t_0 + s (T(t_0, a) - t_0))$ and then viewed as an element of the space ${\tilde \Omega} = C([0, 1)\to X)$. 
Denote by ${\tilde S}(t_0, a)$ the set $S(t_0, a)$ after such re-parametrization. We say that 
the map $a\mapsto S(t_0, a)$ is measurable if,  for any closed subset $F$ of ${\tilde \Omega}$ (with the compact-open topology), 
\[
\{a\in X: {\tilde S}(t_0, a) \cap F \neq \emptyset\}\in {\cal B}_X\,.
\]
\item[{\bf LS3}] If $u\in S(t_0, a)$ and $\tau < T(t_0, a)$, then $T(t_0 + \tau, u(t_0 + \tau; t_0, a)) \ge 
T(t_0, a) - \tau$ and 
$\sigma_\tau u \in S(t_0+\tau, u(t_0+\tau; t_0, a))$. 
\item[{\bf LS4}] If $u\in S(t_0, a)$, $\tau < T(t_0, a)$,  and $v\in S(t_0+\tau, u(t_0+\tau; t_0, a))$, then 
the spliced path 
$w = u \underset{t_0+\tau}{\bowtie} v$, defined by analogy with \eqref{splice}, belongs to the funnel 
$S(t_0, a)$.
\end{enumerate}
\end{definition}

\begin{theorem}\label{thm:2}
Every  family of local abstract integral funnels $S(t_0, a)$, $t_0\in [0, +\infty)$, $a\in X$, with terminal times $T(t_0, a)$, has a selection ${\frak u}(\cdot; t_0, a)$ with the following properties. 
\begin{enumerate}
\item[a)] For every $t_0\ge 0$, the map $ X \owns a \mapsto {\frak u}(\cdot; t_0, a)\in C([t_0, t_0 + T(t_0, a))\to X)$ is measurable. 
\item[b)] ${\frak u}(t_0; t_0, a) = a$,
\item[c)]
${\frak u}(t_2; t_1, {\frak u}(t_1; t_0, a)) = {\frak u}(t_2; t_0, a)$ 
for all $t_1\in [t_0, t_0 + T(t_0, a))$ and $t_2\in [t_1, t_0 + T(t_0, a))$.   
\end{enumerate}
\end{theorem}

\begin{proof} We mimic the proof of Theorem~\ref{thm:main} with a few modifications. 
Let $\vp: X\to [0, 1]$ be a continuous function and let $\lambda$ be a positive real number. 
For $t_0 \ge 0$ and $a\in X$, define the functional $\zeta$ on $S(t_0, a)$ via the formula
\be\label{zeta-2}
\zeta(w) = \int_0^{T(t_0, a)} e^{-\lambda t} \vp(w(t_0 + t))\,dt\,.
\ee
This is a continuous functional and it attains its maximum on the compact set $S(t_0, a)$. 
Denote this maximum by $m_\zeta(t_0, a)$, 
\be\label{m}
m_\zeta(t_0, a) = \max_{w\in S(t_0, a)} \zeta(w)\,
\ee
and define 
\be\label{V}
V_\zeta[S(t_0, a)] = \{v\in S(t_0, a):\; \zeta(v) = m_\zeta(t_0, a)\}\,.
\ee
The way we treat measurability by re-parametrizing the paths (see property {\bf LS2}), allows us to apply 
the measurable maximum theorem, \cite[Theorem 18.19]{A-B}, and conclude that $V_\zeta[S(t_0, a)]$ is a non-empty compact subset of $S(t_0, a)$, and the set-valued map $a \mapsto V_\zeta[S(t_0, a)]$ is measurable. This shows that the family of sets $V_\zeta[S(t_0, a)]$ has properties {\bf LS1} and {\bf LS2}. 
Suppose $u(\cdot; t_0, a)\in V_\zeta[S(t_0, a)]$ and consider the shifted path $\sigma_\tau u$. By property 
{\bf LS3}, $\sigma_\tau u\in S(t_0+\tau, u(t_0 + \tau; t_0, a))$. Let us show that 
$\sigma_\tau u$ maximizes $\zeta$ in the set $S(t_0+\tau, u(t_0 + \tau; t_0, a))$. Pick any path $v$ in 
$S(t_0+\tau, u(t_0 + \tau; t_0, a))$ and consider the spliced path $w = u \underset{t_0+\tau}{\bowtie} v$, which, by property {\bf LS4}, belongs to the funnel 
$S(t_0, a)$. Since $u$ maximizes $\zeta$ over  $S(t_0, a)$, $\zeta(u) \ge \zeta(w)$, i.e., 
\[
\int_0^{T(t_0, a)} e^{-\lambda t} \vp(u(t_0 + t; t_0, a))\,dt \ge \int_0^{T(t_0, a)} e^{-\lambda t} \vp(w(t_0 + t; t_0, a))\,dt\,.
\]
Now compute
\[
\begin{aligned}
& \int_0^{T(t_0, a)} e^{-\lambda t} \vp(w(t_0 + t; t_0, a))\,dt = \\ 
& \int_0^\tau e^{-\lambda t} \vp(u(t_0 + t; t_0, a))\,dt + 
\int_\tau^{T(t_0, a)} e^{-\lambda t} \vp(v(t_0 + t; t_0 + \tau, u(t_0 + \tau; t_0, a))\,dt = \\ 
& \int_0^\tau e^{-\lambda t} \vp(u(t_0 + t; t_0, a))\,dt + e^{-\lambda \tau}\,\int_0^{T(t_0, a) 
- \tau} e^{-\lambda t} \vp(v(t_0 + \tau + t; t_0 + \tau, u(t_0 + \tau; t_0, a) ))\,dt\,,
\end{aligned}
\]
and
\[
\begin{aligned}
& \int_0^{T(t_0, a)} e^{-\lambda t} \vp(u(t_0 + t; t_0, a))\,dt = \\ 
& \int_0^\tau e^{-\lambda t} \vp(u(t_0 + t; t_0, a))\,dt + e^{-\lambda \tau}\,\int_0^{T(t_0, a) - \tau} e^{-\lambda t} \vp(u(t_0 + \tau + t; t_0, a))\,dt\,.
\end{aligned}
\]
Hence, 
\[
\begin{aligned}
& \zeta(\sigma_\tau u) = \int_0^{T(t_0, a) - \tau} e^{-\lambda t} \vp(u(t_0 + \tau + t; t_0, a))\,dt \ge 
 \\ 
& \int_0^{T(t_0, a) 
- \tau} e^{-\lambda t} \vp(v(t_0 + \tau + t; t_0 + \tau, u(t_0 + \tau; t_0, a) ))\,dt  = \zeta(v),
\end{aligned}
\]
which means $\sigma_\tau u$ is a maximizer in $S(t_0+\tau, u(t_0 + \tau; t_0, a))$. With similar modifications, following the proof of Theorem~\ref{thm:main}, one
 verifies property {\bf LS4} for $V_\zeta[S]$. Thus, for any functional $\zeta$ of the form \eqref{zeta-2}, $V_\zeta[S(t_0, a)]$ is a family of abstract local integral funnels.
\medskip

Choose a countable family $\Phi$ of continuous functions $\vp : X\to [0, 1]$ that separates the points of $X$, and choose some 
enumeration $(\lambda_n, \vp_n)$ of the countable set of pairs $(\lambda, \vp)$, where $\lambda$ runs through positive rational numbers and $\vp$ runs through $\Phi$. To each pair $(\lambda_n, \vp_n)$ 
corresponds the functional $\zeta_n$ via \eqref{zeta-2}. Define recursively the shrinking families of 
abstract integral funnels 
$S^0(t_0, a) = S(t_0, a)$, and  $S^{n}(t_0, a) = V_{\zeta_n}[S^{n-1}(t_0, a)], n = 1, 2, \dots $. 
Again, the intersection, 
\[
S^\infty(t_0, a) = \bigcap_{n = 0}^\infty S^{n}(t_0, a)\,,
\] 
is an abstract family of local integral funnels.  
To show that each funnel $S^\infty(t_0, a)$ is a singleton, assume $u, v\in S^\infty(t_0, a)$ Then, for every $\vp\in \Phi$, 
\[
\int_0^{T(t_0, a)}  e^{-\lambda t} \vp(u(t_0 + t; t_0, a))\,dt = \int_0^{T(t_0, a)} e^{-\lambda t} \vp(v(t_0 + t; t_0, a))\,dt\,,\quad \forall \lambda\in \Q_+\,.
\]
By the uniqueness of the Laplace transform, $\vp(u(t_0 + t; t_0, a)) = \vp(v(t_0 + t; t_0, a))$ for all $t\in [0, T(t_0, a))$. 
Because this is true for every $\vp\in\Phi$ and the family $\Phi$ separates the points of $X$, 
we obtain $u(t_0 + t; t_0, a) = v(t_0 + t; t_0, a)$ for all $t\in [0, T(t_0, a))$, i.e., $u = v$ as paths. 
\medskip

For the unique path in the funnel $S^\infty(t_0, a)$, let us use the notation ${\frak u}(t; t_0, a)$. 
If $t_0 < t_1 < T(t_0, a)$, then $(\sigma_{t_1 - t_0} {\frak u})(t) = {\frak u}(t; t_1, {\frak u}(t_1; t_0, a))$ for $t_1\le t < T(t_0, a)$, by property {\bf LS3}.  
Thus, if $t_2 > t_1$, $(\sigma_{t_1 - t_0} {\frak u})(t_2) = {\frak u}(t_2; t_1, {\frak u}(t_1; t_0, a))$. On the other hand, 
$(\sigma_{t_1 - t_0} {\frak u})(t_2) = {\frak u}(t_2; t_0, a)$ by the definition of the shift operator $\sigma_\tau$. 
This establishes the semigroup property of the measurable selection ${\frak u}(t; t_0, a)$. 
The fact that the maps $X\owns a \mapsto {\frak u}(\cdot; t_0, a)\in C([t_0, T(t_0, a))\to X$ are measurable is established by an argument with the distance function similar to the one in the end of the proof of Theorem~\ref{thm:main}. This completes the proof.

\end{proof}


\begin{remark}
There are differential equations for which not all initial conditions $x(t_0)=x_0$ are 
possible: there is a proper subset $C\subset (-\infty, +\infty)\times X$ of allowed initial conditions. 
This is the case, i.e., for the Clairaut equations. We discuss a particular, illustrative example 
in the next section. However, it is not hard to see that Theorems~\ref{thm:main} and \ref{thm:2} 
can be adapted for such situations. 
\end{remark}


\begin{remark}
We can introduce the local process maps $U(t_1, t_0)$ as in the previous section: $U(t_1, t_0)(a) = {\frak u}(t_1; t_0, a)$. However, now $U(t_1, t_0)$ may not be defined on all of $X$ and, for every $a$, the range of admissible $t_1$ will be different. If we restrict $t_0$ to a compact set $[\alpha, \beta]\subset [0, +\infty)$ and $a$  to a compact set $K\subset X$, then, thanks to assumption {\bf TT}, the infimum of 
$ T(t_0, a)$ over $[\alpha, \beta]\times K$ is strictly positive, and this gives a non-trivial  admissible interval for $t_1$ when $(t_0, a)\in [\alpha, \beta]\times K$.  At the moment we do not 
see any benefits in constructing an abstract theory of such local (semi)-processes. 
\end{remark}


\section{Example}\label{sec:ex}

The Clairaut equation is a scalar ODE of the form 
\be\label{eq:cl-gen}
x = t\,{\dot x} + \psi({\dot x})
\ee
with some function $\psi$. The standard method of solution is as follows. 
Differentiate \eqref{eq:cl-gen},
\[
{\dot x} = {\dot x} + t\,{\ddot x}  + \psi^\prime({\dot x})\,{\ddot x}\,,
\]
and simplify the result to obtain
\[
{\ddot x}\,\left(t + \psi^\prime({\dot x})\right) = 0\,.
\]
This offers two possibilities: 
\be\label{eq:ddot}
{\ddot x} = 0
\ee 
and/or 
\be\label{eq:other}
t + \psi^\prime({\dot x}) = 0\,. 
\ee
Solve each of the equations with the initial condition $x(t_0) = x_0$. The first equation 
implies 
\be\label{sol:ddot}
x(t) = x_0 + c\,(t - t_0)
\ee
with a constant $c$. Substitute this into the original equation \eqref{eq:cl-gen}:
\[
x_0 + c\,(t - t_0) = t\,c + \psi(c)\,.
\]
This yields the equation for the possible value(s) of $c$:
\be\label{sol:c}
x_0 = c t_0 + \psi(c)\,.
\ee
Assuming $c(t_0, x_0)$ is a solution of \eqref{sol:c} (there may be many solutions), we can re-write the 
solution \eqref{sol:ddot} as follows:
\be\label{sol:ddot-2}
x(t) = \psi(c(t_0, x_0)) + c(t_0, x_0)\,t\,.
\ee
The second equation \eqref{eq:other} requires some kind of invertibility of the function 
$\psi^\prime$ (the simplest case is when $\psi^{\prime\prime} \neq 0$, i.e., $\psi$ is convex or concave). Assuming the inverse function $(\psi^\prime)^{-1}$ makes sense, we obtain 
\be\label{eq:other-2}
{\dot x} = (\psi^\prime)^{-1}(-t)\,.
\ee 
After integration, 
\be\label{sol:other}
x(t) = x_0 + \int_{t_0}^t (\psi^\prime)^{-1}(-s)\,ds\,.
\ee
Substitute this expression into \eqref{eq:cl-gen}:
\be\label{eq:temp-11}
x_0 + \int_{t_0}^t (\psi^\prime)^{-1}(-s)\,ds = t \,(\psi^\prime)^{-1}(-t) + \psi\left((\psi^\prime)^{-1}(-t)\right)
\ee
Notice that
\[
\frac{d\hfil}{ds} \left[s \,(\psi^\prime)^{-1}(-s) + \psi\left((\psi^\prime)^{-1}(-s)\right)\right] = 
(\psi^\prime)^{-1}(-s)
\]
Thus, equation \eqref{eq:temp-11} can be simplified to
\be\label{eq:restr-2}
x_0 - \left[t_0 \,(\psi^\prime)^{-1}(-t_0) + \psi\left((\psi^\prime)^{-1}(-t_0)\right)\right] = 0
\ee
It is useful to introduce the following version of the Legendre transform of the function $\psi$:
\be\label{Leg}
{\tilde \psi}(t_*) = \inf_t\;\;[\;t_*\cdot t + \psi(t)\;]\,.
\ee
Then equation \eqref{eq:restr-2} can be written in the form 
\be\label{eq:restr-21}
x_0 = {\tilde \psi}(t_0)\,,
\ee
and the solution \eqref{sol:other} takes the form
\[
x(t) = x_0 + {\tilde \psi}(t) - {\tilde \psi}(t_0)\,. 
\]
In view of \eqref{eq:restr-21}, we obtain
\be\label{sol:other-2}
x(t) = {\tilde \psi}(t)\,.
\ee
This is the so-called singular solution. We see that the singular solution corresponds to the 
 value of $c = c(t_0, x_0)$ in \eqref{sol:ddot-2} that minimizes $[c\cdot t + \psi(c)]$. 
 \bigskip
 
Consider the special case $\psi(s) = s^2$, i.e., consider the equation
\be\label{eq:square}
x = t \,{\dot x} + ({\dot x})^2\,.
\ee
The corresponding equation \eqref{sol:c} has two solutions for $c$,
\be\label{c_0}
c_{\pm}(t_0, x_0) = - \frac{t_0}{2} \pm \sqrt{\frac{t_0^2}{4} + x_0}\,,
\ee
provided $x_0 + {t_0^2}/{4} > 0$, has one solution when $x_0 + {t_0^2}/{4} = 0$, 
and has no solutions when $x_0 + {t_0^2}/{4} < 0$. Since ${\tilde \psi}(t_*) = - t_*^2/4$, the 
singular solution is  
\be\label{sol:sing}
x(t) = - \frac{t^2}{4}\,.
\ee
Thus, the region $C\subset (-\infty,+\infty)\times \R$ of the allowed initial data $(t_0, x_0)$ 
is the parabola $x = - t^2/4$ and the points above it:
\[
C = \{(t, x): x \ge - \frac{t^2}{4}\}\,.
\]
The parabola is the envelope of 
the straight lines \eqref{sol:ddot} or, equivalently, \eqref{sol:ddot-2}:
\be\label{line}
x(t) = c^2 + c\,t\,.
\ee 
Through every point $(t_0, x_0)$ in the interior of $C$ pass two straight lines (solutions) corresponding to 
$c_+(t_0, x_0)$ and $c_-(t_0, x_0)$. It is useful to notice that one of the lines \eqref{line} touches the parabola in the past (at some moment $t = {\mathfrak t}_p(t_0, x_0) < t_0$) 
and the other touches the parabola in the future (at some moment $t = {\mathfrak t}_f(t_0, x_0) > t_0$). Let us denote these lines (solutions) ${\mathfrak x}_p(t; t_0, x_0)$ and ${\mathfrak x}_f(t; t_0, x_0)$, 
respectively. They are defined for all $t\in\R$. 
The integral funnel 
$S(t_0, x_0)$ contains the rays ${\mathfrak x}_p(t; t_0, x_0)$ and ${\mathfrak x}_f(t; t_0, x_0)$ for $t\ge t_0$. In addition, it contains infinitely many solutions that branch off ${\mathfrak x}_f(t; t_0, x_0)$: once the line touches the parabola at $t = {\mathfrak t}_f(t_0, x_0)$, the trajectory 
may continue along the parabola forever, or, at any time $r > {\mathfrak t}_f(t_0, x_0)$ it may take off 
along the tangent line. 

To every point $(t_0, x_0)$ on the boundary of $C$ (on the parabola) correspond 
infinitely many solutions of \eqref{eq:square} that can be divided into three categories: 1) there is the singular solution \eqref{sol:sing}; 
2) there is a ray tangent to the parabola: 
\be\label{tan}
{\mathfrak x}_{pf}(t; t_0, - \frac{t_0^2}{4}) = - \frac{t_0}{2}\,t + \frac{t_0^2}{4}\,,\quad t\ge t_0\;;
\ee
3) there is a one-parameter family of solutions that follow the parabola for some time and then get off on the tangent line: 
\be\label{branch}
x_r(t) = \begin{cases}
- \frac{t^2}{4} & \text{for}\quad t_0\le t \le r \\ 
- \frac{r}{2}\,t + \frac{r^2}{4} & \text{for}\quad t\ge r\,.
\end{cases}
\ee
All these solutions form the funnel $S(t_0, x_0)$ when $(t_0, x_0)$ is on the parabola. 

It is not hard to see that the trajectories that stay on the parabola for some time and then leave it
cannot be part of the semi-process. Thus, 
 there are three choices of a semi-process corresponding to 
equation \eqref{eq:square}. The first choice is completely determined by the initial conditions 
on the parabola: for $(t_*, x_*)\in\partial C$, 
\be
{\mathfrak u}(t; t_*, x_*) = - \frac{t_*}{2}\,t + \frac{t_*^2}{4}\quad \forall t\ge t_*\,.
\ee
When the point $(t_*, x_*)$ slides along the parabola clockwise, the corresponding rays ${\mathfrak x}_{pf}(t; t_*, x_*) = {\mathfrak u}(t; t_*, x_*)$ sweep 
the interior of $C$. For every point $(t_0, x_0)$ in the interior of $C$, there is a unique point 
$(t_*, x_*)$ (with $t_* = {\mathfrak t}_p(t_0, x_0)$) on the parabola such that ${\mathfrak u}(t_0; t_*, x_*) = x_0$. As a consequence, 
${\mathfrak u}(t; t_0, x_0) = {\mathfrak u}(t; t_*, x_*)$ describes the ray-solution starting 
at $(t_0, x_0)$. 

The second choice of the semi-process coincides with the first one on the interior of $C$, while 
on the parabola we pick the singular solution:  ${\mathfrak u}(t; t_0, - t_0^2/4) = - t^2/4$. 
 
The third choice is to select the trajectories (rays) that touch the parabola in the future 
and then continue along the parabola. In other words, 
\be
{\mathfrak u}(t; t_0, x_0) = - \frac{t^2}{4}\quad\text{if}\quad x_0 = - \frac{t_0^2}{4}\,,
\ee
and, if $(t_0, x_0)$ is in the interior of $C$, then  
\be
{\mathfrak u}(t; t_0, x_0) =
\begin{cases}
{\mathfrak x}_f(t; t_0, x_0) & t_0\le t\le {\mathfrak t}_f(t_0, x_0)\,,\\
- \frac{t^2}{4} & t > {\mathfrak t}_f(t_0, x_0)\,.
\end{cases}
\ee
All three semi-processes are Borel measurable. 
Depending on the choice of a countable family $\Phi$ of continuous bounded functions that separate 
the points of $X = \R$, the maximization procedure described in the proof of Theorem~\ref{thm:main} 
will pick out one of the three semi-processes.


\end{document}